\documentclass [12pt,letterpaper,reqno]{amsart}

\DeclareMathSymbol{\twoheadrightarrow} {\mathrel}{AMSa}{"10}

\def\Q{{\mathbb Q}}
\def\Z{{\mathbb Z}}

\def\F{{\mathbb F}}

\def\Gal{\mathrm{Gal}}

\def\Perm{\mathrm{Perm}}

\def\End{\mathrm{End}}

\def\Aut{\mathrm{Aut}}
\def\Hom{\mathrm{Hom}}

    \def\RR{\mathfrak{R}}

\def\fchar{\mathrm{char}}

\def\dim{\mathrm{dim}}

\newtheorem{thm}{Theorem}[section]
\newtheorem{lem}[thm]{Lemma}

\newtheorem{prop}[thm]{Proposition}
\theoremstyle{definition}

\newtheorem{rem}[thm]{Remark}

\title[Non-isogenous  hyperelliptic jacobians]
{Non-isogenous   elliptic curves and hyperelliptic jacobians II}
\author[Yuri G. Zarhin]{Yuri G. Zarhin}
\address{Department of Mathematics, Pennsylvania State University,
University Park, PA 16802, USA}
\email{zarhin\char`\@math.psu.edu}
\dedicatory{To Yuri Ivanovich Manin on the occasion of his 85th birthday}
\thanks{The author  was partially supported by Simons Foundation Collaboration grant   \# 585711.
This work was done during his stay in 2022 at the Max-Planck Institut f\"ur Mathematik (Bonn, Germany), whose hospitality and support are gratefully acknowledged.}

\begin{document}
\begin{abstract}
Let $K$ be a field of characteristic different from $2$, $\bar{K}$ its algebraic closure.
Let $n \ge 3$ be an odd integer.
Let $f(x)$ and $h(x)$ be degree $n$ polynomials with coefficients in $K$ and without repeated roots.
Let us consider genus $(n-1)/2$ hyperelliptic curves
$C_f: y^2=f(x)$ and $C_h: y^2=h(x)$, and their jacobians $J(C_f)$ and $J(C_h)$, which are
$(n-1)/2$-dimensional abelian varieties defined over $K$. 

Suppose that one of the
polynomials is irreducible and the other splits completely over $K$.
We prove that if $J(C_f)$ and $J(C_h)$ are  isogenous over $\bar{K}$ then there is an (odd) prime $\ell$
dividing $n$ such that both  endomorphism algebras of $J(C_f)$ and  of $J(C_h)$ contain
a subfield that  is isomorphic to
the field of $\ell$th roots of $1$.

\end{abstract}

\subjclass[2010]{14H40, 14K05, 11G30, 11G10}
\keywords{hyperelliptic curves, jacobians, isogenies of abelian varieties}

\maketitle
\section{Definitions, notations, statements}
This paper is a follow up of \cite{ZarhinMRL22} and we use its notation. (See also \cite{ZarhinSh03,ZarhinMZ06}.)
In particular, if $f(x)\in K[x]$ is a polynomial of odd degree $n=2g+1$  without repeated roots and with coefficients
in a field $K$, and with $\mathrm{char}(K)\ne 2$,
 then we write $C_f$ for the smooth projective model of the plane curve $y^2=f(x)$
and $J(C_f)$ for its jacobian, which is a $g$-dimensional abelian variety over $K$. 
We fix an  algebraic closure  $\bar{K}$ of $K$ and write $\Gal(K)=\Aut(\bar{K}/K)$ for the group of its $K$-linear automorpisms.
We write $\RR_f\subset \bar{K}$ for the $n$-element set of roots of $f(x)$,  $K(\RR_f)$ for the splitting field of $f(x)$ and
$\Gal(f/K)$ for the Galois group
$$\Gal(K(\RR_f)/K)=\Aut(K(\RR_f)/K)$$
of $f(x)$. As usual, one may view $\Gal(f/K)$ as a certain permutation subgroup of the group $\Perm(\RR_f)$ of all permutations of $\RR_f$.


Throughout this paper, $n\ge 3$ is an odd integer, $f(x)$ and $h(x)$ are degree $n$ polynomials with coefficients in $K$ and without repeated roots,
$$C_f: y^2=f(x), \ C_h: y^2=h(x)$$
are the corresponding genus $(n-1)/2$ hyperelliptic curves over $K$, whose jacobians we denote by $J(C_f)$ and $J(C_h)$, respectively. These jacobians are $(n-1)/2$-dimensional abelian varieties defined over $K$.

The main result of this paper is the following assertion.

\begin{thm}
\label{endoH}
Suppose that $n\ge 3$ is an odd prime.
Let $K$ be a field of characteristic different from
$2$. 
Let $f(x), h(x) \in K[x]$ be degree $n$ polynomials without repeated roots.  Suppose that one of the
polynomials is irreducible and the other is reducible.

If the corresponding hyperelliptic jacobians $J(C_f)$ and $J(C_h)$ are  isogenous over $\bar{K}$ then they both
are abelian varieties of CM type over $\bar{K}$ with multiplication by  the $n$th cyclotomic field $\Q(\zeta_n)$.
\end{thm}

\begin{rem}
Theorem \ref{endoH} was proven in \cite{ZarhinMRL22} under the additional assumption that $2$ is a primitive root modulo $n$.

\end{rem}

The next assertion may be viewed as a  partial generalization of  Theorem \ref{endoH} to the case of arbitrary odd $n$.

\begin{thm}
\label{endoH2}
Suppose that $n\ge 3$ is an odd integer.
Let $K$ be a field of characteristic different from
$2$. 
Let $f(x), h(x) \in K[x]$ be degree $n$ polynomials without repeated roots.  Suppose that $f(x)$ is irreducible over $K$.

Assume additionally that 
the order of the Galois group $\Gal(h/K)$ of $h(x)$ is prime to $n$.
(E.g., each irreducible factor of $h(x)$  over $K$ has degree $1$ or $2$.)

If the corresponding hyperelliptic jacobians $J(C_f)$ and $J(C_h)$ are  isogenous over $\bar{K}$ then 
there is a prime divisor $\ell$ of $n$ such that both endomorphism algebras $\End^0(J(C_f))$ and  $\End^0(J(C_h))$
contain an invertible element of multiplicative order $\ell$.
\end{thm}

\begin{rem}
Assume that the conditions of  Theorem \ref{endoH2} hold. Then $h(x)$ is reducible over $K$.
Indeed, if $h(x)$ is irreducible then $\Gal(h/K)$ acts {\sl transitively} on the $n$-element set $\mathcal{R}_h$
of roots of $h(x)$.
Therefore its order $\#(\Gal(h/K))$ is divisible by $n$,
which gives us a desired contradiction.
\end{rem}

The paper is organized as follows. In Section \ref{order2} we recall basic facts about Galois properties of points of order 2 and 4 on abelian varieties and hyperelliptic jacobians.
We also state Theorem \ref{isogEll} that is  a stronger (but a more technical) version of Theorem \ref{endoH2}.
Section \ref{mainproof} contains the proof of Theorem \ref{isogEll}.
(Notice that Lemma \ref{isogenyNotInv} plays a crucial role in the proof and may be of certain independent interest.)
We prove Theorem \ref{endoH}  in Section \ref{PendoH}.

{\bf Acknowledgements}. I am deeply grateful to the referee, whose comments helped to improve the exposition.

\section{Points of order 2 and 4 on hyperelliptic jacobians}
\label{order2}
Let $K_s \subset \bar{K}$ be the separable algebraic closure of $K$.
Let  $X$  be a positive-dimensional abelian variety   over $K$. If 
$d$ is a positive integer  then we write $X[d]$ for the kernel of
multiplication by $d$ in $X(\bar{K})$.  Recall (\cite[Sect. 6]{Mumford}, \cite[Sect. 8, Remark 8.4]{Milne}) that  if $d$ is {\sl not} divisible
by $\fchar(K)$ then  $X[d]$ is a $\Gal(K)$-submodule of $X(K_s)$;
in addition, $X[d]$  is isomorphic as a commutative group to $(\Z/d\Z)^{2\dim(X)}$.

Let  $K(X[d])$ be the {\sl field of definition} of all torsion points of order dividing $d$ on
$X$.  It is well known \cite[Remark 8.4]{Milne} that $K(X[d])$ lies in $K_s$ and  is a finite Galois extension of $K$.
Let us put
$$\tilde{G}_{d,X}:=\Gal(K(X[d])/K).$$
One may view $\tilde{G}_{d,X}$ as a certain subgroup of $\Aut_{\Z/d\Z}(X[d])$ and $X[d]$ as a faithful $\tilde{G}_{d,X}$-module.
In addition, the structure of the $\Gal(K)$-module on $X[d]$ is induced by the canonical (continuous) {\sl surjective} group homomorphism
$$\tilde{\rho}_{d,X}:\Gal(K)\twoheadrightarrow \Gal(K(X[d])/K)=\tilde{G}_{d,X}.$$

For example, if $d=2\ne \fchar(K)$ then
$X[2]$ is a $2\dim(X)$-dimensional vector space over the $2$-element prime
field $\F_{2}=\Z/2\Z$ and the inclusion $\tilde{G}_{2,X}
\subset \Aut_{\F_{2}}(X[2])$ defines a {\sl faithful} linear
representation of the group  $\tilde{G}_{2,X}$ in the vector
space $X[2]$ over $\F_2$.

\begin{rem}
\label{silver}
\begin{itemize}
\item[(i)]
Let $K(X[4])$ be the field of definition of all points of order dividing $4$ on $X$. 
It is well known that $K(X[4])/K$ is a finite Galois field extension, 
$$K\subset K(X[2])\subset K(X[4])\subset K_s$$
and the Galois group $\Gal(K(X[4])/K(X[2]))$ is a finite commutative group of exponent $2$ or $1$
(e.g., see \cite[Remark 2.2(i)]{ZarhinMRL22}).

\item[(ii)]
Let $d \ge 3$ be an integer that is {\sl not} divisible by $\fchar(K)$. By a result of A. Silverberg,
\cite[Th. 2.4]{Silverberg} 
all the endomorphisms of $X$ are defined over $K(X[d])$. (See \cite{GK,Remond,Pip} for further results
concerning the field of definition of all  endomorphisms of an abelian variety.)
\end{itemize}
\end{rem}

 The following assertion contains Theorem \ref{endoH2} as a special case (when $Y=J(C_h)$).

\begin{thm}
\label{isogEll}
Let 
$K$ be a field of characteristic $\ne 2$. Let $g$ be a positive unteger, $n:=2g+1$, and $f(x)\in K[x]$ a degree  $n$ irreducible polynomial
without repeated roots.  Let us put $X=J(C_f)$.

Let $Y$ be a $g$-dimensional abelian variety  over $K$ such that the order of $\tilde{G}_{2,Y}$ is prime to $n$.
(E.g., $K(Y[2])=K$.)

Suppose that $X$ and $Y$ are isogenous over $\bar{K}$.

Then there is an odd prime $\ell$ dividing $n$ such that both endomorphism algebras $\End^0(X)$ and $\End^0(Y)$ contain 
an invertible element of multiplicative order $\ell$.

In addition, if $n=2g+1$ is a prime then
both $X$ and $Y$ are abelian varieties of CM type over $\bar{K}$ with multiplication by the $n$th cyclotomic field $\Q(\zeta_{n})$.

\end{thm}

We will prove  Theorem \ref{isogEll} in Section \ref{mainproof}. Our proof is based on the Galois properties of points of order $2$ on hyperelliptic jacobians
that will be discussed in the next subsection.

\subsection{Galois properties}
\label{QRR}
In  this subsection we recall  a  well known  explicit description of the Galois module $J(C_f)[2]$  \cite{Mori2,ZarhinTexel}
for arbitrary separable $f(x)$ and odd $n$. 
Let us start with the $n$-dimensional $\F_2$-vector space
$$\F_2^{\RR_f}=\{\phi:\RR_f \to \F_2\}$$
of all $\F_2$-valued functions on $\RR_f$. The  action of $\Perm(\RR_f)$ on $\RR_f$ provides $\F_2^{\RR_f}$ with the structure of a faithful $\Perm(\RR_f)$-module, which splits into a direct sum
\begin{equation}
 \label{splitting}
\F_2^{\RR_f}=\F_2\cdot {\bf 1}_{\RR_f}\oplus Q_{\RR_f}
\end{equation}
of the one-dimensional subspace $\F_2\cdot {\bf 1}_{\RR_f}$ of constant functions  and the $(n-1)$-dimensional {\sl heart} \cite{Klemm,Mortimer}
$$Q_{\RR_f}:=\{\phi:\RR_f \to \F_2\mid \sum_{\alpha\in\RR_f}\phi(\alpha)=0\}$$
(here we use that $n$ is odd). Clearly, the  $\Perm(\RR_f)$-module is faithful. It remains faithful if we view it as a $\Gal(f/K)$-module.  
There is a $\Perm(\RR_f)$-invariant $\F_2$-bilinear pairing
$$\Psi: \F_2^{\RR_f} \times \F_2^{\RR_f} \to \F_2, \  \phi,\psi \mapsto \sum_{\alpha\in \RR_f}\phi(\alpha)\psi(\alpha)$$
and the splitting \eqref{splitting} is an orthogonal (with respect to $\Psi$) direct sum. Clearly, the restriction of $\Psi$ to $\F_2\cdot {\bf 1}_{\RR_f}$ is nondegenerate and therefore
the restriction of $\Psi$ to $Q_{\RR_f}$ is nondegenerate as well. This implies that the  $\Gal(f/K)$-modules $Q_{\RR_f}$ and its  {\sl dual}
$\Hom_{\F_2}(Q_{\RR_f},\F_2)$ are  {\sl isomorphic}.

The field inclusion $K(\RR_f)\subset K_s$ induces the {\sl surjective} continuous homomorphism
$$\Gal(K)=\Gal(K_s/K)\twoheadrightarrow \Gal(K(\RR_f)/K)=\Gal(f/K),$$
which gives rise to the natural structure of the $\Gal(K)$-module on $Q_{\RR_f}$ such that the image of $\Gal(K)$ in $\Aut_{\F_2}(Q_{\RR_f})$ coincides with
$$\Gal(f/K)\subset \Perm(\RR_f)\hookrightarrow \Aut_{\F_2}(Q_{\RR_f}).$$  This implies that the {\sl Galois modules $Q_{\RR_f}$ and $\Hom_{\F_2}(Q_{\RR_f},\F_2)$ are  isomorphic}.

It is well known (see, e.g., \cite{Mori2,ZarhinTexel}) that the $\Gal(K)$-module $J(C_f)[2]$ and $Q_{\RR_f}$ are canonically  isomorphic. 
This implies that the groups $\tilde{G}_{2,J(C_f)}$ and $\Gal(f)$ are canonically isomorphic. It is also clear that $K(\RR_f)$ coincides with  $K(J(C_f)[2])$.
We will need the following  assertion.

\begin{lem}
\label{invTran}
Suppose that $f(x)$ is irreducible over $K$. Then:

\begin{itemize}
 \item [(i)]

$Q_{\RR_f}$ does not contain nonzero Galois-invariants.
\item [(ii)]
Every Galois-invariant linear functional $Q_{\RR_f} \to \F_2$ is zero.
\item [(iii)]
Let $W$ be a $\F_2$-vector space provided with the trivial action of $\Gal(K)$.
Then every homomorphism of the Galois modules $Q_{\RR_f} \to W$ is zero and every
homomorphism of the Galois modules $J(C_f)[2] \to W$ is zero as well.
\end{itemize}
\end{lem}

\begin{proof}
Recall that the irreducibility means that the Galois group acts transitively on $\RR_f$.
Let $\phi \in Q_{\RR_f}$ be a  Galois-invariant function on $\RR_f$. The transitivity implies that $\phi$ is   constant. If $\phi$ is not (identically) zero then
$\phi(\alpha)=1$ for all $\alpha\in \RR_f$ and therefore (since $\phi \in Q_{\RR_f}$)
$$0=\sum_{\alpha\in \RR_f}\phi(\alpha)=n \cdot 1=1 \in \F_2,$$ 
i.e., $0=1$,
 which is absurd. The obtained contradiction proves
that $\phi \equiv 0$.
In order to prove the second assertion of Lemma, recall that  the Galois modules $Q_{\RR_f}$ and $\Hom_{\F_2}(Q_{\RR_f},\F_2)$ are isomorphic. 
Now the second assertion of the lemma follows from the already proven first one.
On the other hand, the third assertion is an immmediate corollary of the second one: one has only to choose a basis of $W$ and recall that the Galois modules 
 $Q_{\RR_f}$ and  $J(C_f)[2]$ are isomorphic.
\end{proof}

\section{Isogenous hyperelliptic jacobians}
\label{mainproof}

We will deduce Theorem \ref{isogEll} from the following auxiliary statements.

\begin{lem}
 \label{orbits}
 Let $G$ be a transitive permutation group of a finite nonempty set $\RR$, and $H$ a normal subgroup of $G$.
 Then the number of $H$-orbits in $\RR$ divides both $\#(\RR)$ and the index $(G:H)$. In particular, if 
 $\#(\RR)$ and $(G:H)$ are relatively prime then $H$ acts transitively on $\RR$.
 
 \end{lem}

\begin{lem}
 \label{remainIrr} Let $f(x)$ be a degree $n$ irreducible polynomial over a field $K$ and without repeated roots.
 Let $K_1/K$ be a finite Galois field extension, whose degree is prime to $n$.
 Then $f(x)$ remains irreducible over $K_1$. In particular, the order of Galois group $\Gal(f/K_1)$ is divisible by $n$.
 \end{lem}

 \begin{lem}
 \label{isogenyNotInv}
 We keep the notation and assumptions of Theorem \ref{isogEll}.
 
 Suppose that  $K=K(Y[4])$.  Then there is a nontrivial group homomorphism
 $$\chi: \Gal(K(X[4])/K) \to \End^0(Y)^{*},$$
 whose image
 $$\Gamma:=\mathrm{Im}(\chi)\subset \End^0(Y)^{*}$$
 is a finite group that enjoys the following property.  
 
 The integers $n$ and  $\#(\Gamma)$ are not relatively prime.
 In other words,
 there is an odd prime $\ell$ that divides both $n$ and  $\#(\Gamma)$.
  \end{lem}

 \begin{proof}[Proof of Theorem \ref{isogEll} (modulo Lemmas \ref{remainIrr} and \ref{isogenyNotInv})]
 
 The degree satisfies
$$[K(Y[4]):K]=[K(Y[4]):K(Y[2])] \cdot [K(Y[2]):K].$$
By assumption we know that $[K(Y[2]):K]=\#(\tilde{G}_{2,Y})$ is prime to $n$. Thanks to Remark \ref{silver}(i), $[K(Y[4]):K(Y[2])]$ is either $1$ or a power of $2$.
This implies that the product $[K(Y[4]):K]$ is also prime to the odd integer $n$.
In light of Lemma \ref{remainIrr}, $f(x)$ remains irreducible over the field $K_1=K(Y[4])$. Replacing $K$ by $K_1$, we may assume
that $K=K(Y[4])$. 

By Lemma \ref{isogenyNotInv}, there is a nontrivial finite subgroup 
$\Gamma\subset \End^0(Y)^{*}$,
whose order is divisible by a certain
 prime divisor $\ell$ of $n$. Then $\Gamma$ contains an element $u$ of order $\ell$ that is an invertible element in $\End^0(Y)$ of multiplicative
order $\ell$. Hence, the $\Q$-subalgebra $\Q[u]$ of
$\End^0(Y)$ generated by $u$ is isomorphic to a quotient of  the direct sum $\Q\oplus \Q(\zeta_{\ell})$.  Since $\ell$ is odd, 
$\Q[u]$ is isomorphic 
either to  $\Q(\zeta_{\ell})$ or to  $\Q\oplus \Q(\zeta_{\ell})$.

If $n$ is a prime then $\ell=n=2g+1$ and the same arguments as in \cite[Sect. 4, proof of Prop. 2.4]{ZarhinMRL22}  (based
on \cite[Ch. II, Prop. 1]{Shimura})
show that the $\Q$-subalgebra 
$\Q[u]$ of $\End^0(Y)$  is isomorphic to $\Q(\zeta_n)$ and therefore $Y$ (and also $X$) is an abelian variety of CM type
with multiplication by $\Q(\zeta_n)$.

\end{proof}

\begin{proof}[Proof of Lemma \ref{orbits}]
Let us put
$$n=\#(\RR), \ h=(G:H).$$
Since $G$ acts transitively on the $n$-element set $\RR$, all the orbits of  the normal subgroup $H$  have the same cardinality, say, $r$ \cite[Prop. 4.4 on p. 22]{Passman}.
This implies that $m=n/r$ is the {\sl number of all $H$-orbits} in $\RR$. In particular, both $m$ and $r$ divide $n$.

 Let $B$ be an $r$-element orbit of $H$ in $\RR$.
Then $B$ is a {\sl subset of imprimitivity} for $G$ \cite[Proof of Prop. 4.4 on p. 22]{Passman}.  Let $O$ be the   $m$-element set of $H$-orbits in $\RR$.
Now the transitivity of the action of $G$ on $\RR$ and the normality of $H$ implies that $G$ acts transitively on $O$ and this action factors through the {\sl transitive} action
of
the  {\sl quotient} $G/H$  on $O$.
This implies that $m=\#(O)$ divides $\#(G/H)=(G:H)=h$.  So, $m$ divides both $n$ and $h$, which ends the proof.
\end{proof}

\begin{proof}[Proof of Lemma \ref{remainIrr}]
Since $K_1/K$ is Galois, the group $\Gal(f/K_1)$ is a {\sl normal} subgroup of $\Gal(f/K)$, whose index $h$ divides $[K_1:K]$ and therefore is also prime to $n$.
It follows from  Lemma \ref{orbits} applied to
$$\RR=\RR_f, \ G=\Gal(f/K), \ H=\Gal(f/K_1)$$
that $\Gal(f/K_1)$  acts transitively on $\RR_f$, i.e., $f(x)$ is irreducible over $K_1$. 
\end{proof}

\begin{proof}[Proof of Lemma \ref{isogenyNotInv}]
In light of the theorem of Silverberg (Remark \ref{silver}(ii)), all endomorphisms of $Y$ are defined over $K$.
Applying the theorem of Silverberg (see Remark \ref{silver}(ii) above) to $X\times Y$, we conclude that all the homomorphisms from $X$ to $Y$ are defined over $K(X[4])$.

Let $\mu: X \to Y$ be an isogeny. Dividing, if necessary, $\mu$ by a suitable power of $2$, we may and will assume that
\begin{equation}
\label{notZero}
\mu(X[2]) \ne \{0\}.
\end{equation}
Let us put
$$G_4:=\Gal(K(X[4])/K), \ G=\Gal(K(X[2])/K)=\Gal(f/K).$$
We know that $\mu$ is defined over $K(X[4])$. This allows us to define for each $\sigma \in G_4$ the isogeny
$\sigma(\mu):X \to Y$, which is the Galois-conjugate of $\mu$ (recall that both $X$ and $Y$ are defined over $K$). 
Then the same construction as in \cite[Sect. 4, proof of Prop. 2.4]{ZarhinMRL22} allows us to define a map
$$c: G_4 \to \End^0(Y)^{*}, \ \sigma \mapsto c(\sigma)$$
where $c(\sigma)$ is determined by
$$\sigma(\mu)=c(\sigma)\mu \ \forall \sigma \in G_4=\Gal(K(X[4])/K).$$
We have for each $\sigma,\tau \in G_4$
$$c(\sigma\tau)\mu=\sigma\tau(\mu)=\sigma(\tau(\mu))= \sigma(c(\tau)\mu)=c(\tau) \sigma(\mu)=c(\tau)c(\sigma)\mu$$
(here we use that all elements of $\End(Y)$ are defined over $K$, i.e., are $G_4$-invariant).
Therefore
$$c(\sigma\tau)=c(\tau)c(\sigma) \ \forall \sigma, \tau \in G_4=\Gal(K(X[4])/K).$$
This means that the map
$$\chi: G_4=\Gal(K(X[4])/K) \to \End^0(Y)^{*}, \ \sigma \mapsto \chi(\sigma)=c(\sigma)^{-1}$$
is a {\sl group homomorphism}. Let $\Gamma \subset \End^0(Y)^{*}$ be the image of $\chi$, which is a finite
subgroup of $\End^0(Y)^{*}$.  We need to check that there is a prime divisor $\ell$ of $n$ that divides $\#(\Gamma)$.

Let $H_4$ be the kernel of $\chi$, i.e.,
\begin{equation}
\label{H4}
H_4=\{\sigma \in G_4\mid \sigma(\mu)=\mu\}.
\end{equation}
By definition, $H_4$ is a normal subgroup of $G_4$. Let $H$ be the image of $H_4$ under the {\sl surjective} group homomorphism
$$G_4=\Gal(K(X[4])/K)\twoheadrightarrow \Gal(K(X[2])/K)=\Gal(f/K)=G.$$
The surjectiveness implies that $H$ is a normal subgroup of $G$ and the index $(G:H)$ divides 
$$(G_4:H_4)=\#(\Gamma).$$

In order to finish the proof, we  need the following assertion that will be proven at the end of this section.

\begin{prop}
\label{GneH}
The subgroup $H$ of $G$ is not transitive on $\RR_f$. 
\end{prop}

{\bf End of Proof of Lemma \ref{isogenyNotInv} (modulo Proposition \ref{GneH})}
Combining Proposition \ref{GneH} with Lemma \ref{orbits}, we conclude that $(G:H)$ is {\sl not} prime to $n$.
Hence, there is a prime $\ell$ that divides both $(G:H)$ and $n$ (recall that $n$ is odd). Since $(G:H)$ divides 
$(G_4:H_4)$, we conclude thast $\ell$ divides $(G_4:H_4)=\#(\Gamma)$, which ends the proof.

\end{proof}

 \begin{proof}[Proof of  Proposition \ref{GneH}]
 Suppose that $H$ is   transitive. Then $f(x)$ remains {\sl irreducible} over the (sub)field $K(X[2])^{H}$ of $H$-invariants.
  Replacing $K$ by its overfield $K(X[4])^{H_4}$, we may and will assume that
  $$H_4=\Gal(K(X[4])/K), \ H=\Gal(K(X[2])/K)=\Gal(f/K).$$
  In particular,  
  \begin{equation}
  \label{muGalois}
  \sigma(\mu)=\mu\  \forall \sigma \in H_4=\Gal(K(X[4])/K).
  \end{equation}
  Recall that
  \begin{equation}
  \label{sigmaX4}
  \sigma(\mu)(\sigma(x))=\sigma(\mu(x))\  \forall \sigma \in \Gal(K(X[4])/K), \
  x \in X(K[4]).
  \end{equation}
  Combining \eqref{sigmaX4} with \eqref{muGalois} and taking into account that 
  $$X[2]\subset X[4] \subset X(K[4]),$$
  we obtain that
  \begin{equation}
  \label{sigmaX2}
  \mu(\sigma(x))=\sigma(\mu(x))\  \forall \sigma \in\Gal(K(X[4])/K), \
  x \in X(K[2]).
  \end{equation}
  Combining \eqref{notZero} with \eqref{sigmaX2}, we obtain that $\mu$ induces a {\sl nonzero} homomorphism of Galois modules
  $$X[2] \to Y[2].$$
  Notice that the Galois module $Y[2]$ is {\sl trivial}, because 
  $$K \subset K(Y[2])\subset K(Y[4])=K,$$
  i.e., $K(Y[2])=K$. On the other hand, the irreducibility of $f(x)$ over $K$ implies (thanks to Lemma \ref{invTran}(iii))
  that every homomorphism of the Galois module $X[2]$ to the trivial Galois module $Y[2]$ is zero. The obtained contradiction proves that $H$ is not transitive. 
\end{proof}

\section{Proof of  Theorem \ref{endoH}}
\label{PendoH}
So, $n$ is an odd prime, both $f(x)$ and $h(x)\in K[x]$ are degree $n$ polynomials without repeated roots,
$f(x)$ is irreducible and $h(x)$ is reducible. Since $n$ is a prime,  the reducibility of $h(x)$ implies that 
the order of $\Gal(h/K)$ is prime to $n$ (see \cite[Lemma 2.6]{ZarhinMRL22}).
This implies that if we put $Y=J(C_h)$ then the group  $\tilde{G}_{2,Y}=\Gal(h/K)$ has order that is prime to $n$. 
Now the desired result follows readily from Theorem \ref{isogEll} (applied to $\ell=n$).

\end{document}